\documentclass{article}
\usepackage[T2A]{fontenc}
\usepackage[cp1251]{inputenc}
\usepackage[russian,english]{babel}
\usepackage[tbtags]{amsmath}
\usepackage{amsfonts,amssymb,mathrsfs,amscd,msb-a,amsmath}
\JournalName{}
\numberwithin{equation}{section}
\theoremstyle{plain}

\newtheorem{theorem}{Theorem}
\theoremstyle{plain}

\newtheorem{lemma}{Lemma}
\newtheorem{corollary}{Corollary}
\theoremstyle{definition}
\newtheorem{proof}{Proof}

\newcommand{\mlegendre}[2]{\left(\frac{#1}{#2}\right)}
\begin{document}

\title{Quadratic characters with positive partial sums}
\author{A.\,B.~Kalmynin}
\address{Steklov Mathematical Institute of Russian Academy of Sciences, Moscow, Russia
}
\email{alkalb1995cd@mail.ru}
\date{}
\udk{}
\maketitle
\footnotetext{The work is supported by the Russian Science Foundation under grant \textnumero 19-11-0001.}
\begin{abstract}\begin{bf}{Abstract.}\end{bf}
Let $\mathcal L^+$ be the set of all primes $p$ for which the sums of $\mlegendre{n}{p}$ over the interval $[1,N]$ are non-negative for all $N$. We prove that the estimate
\[
|\mathcal L^+\cap [1,x]|\ll \frac{x}{\ln x(\ln\ln x)^{c-o(1)}}
\]
holds for $c\approx 0.0368$
\end{abstract}
\section{Introduction}

Let $p$ be a prime number. The only quadratic character modulo $p$ is the Legendre symbol $\chi_p(n)=\mlegendre{n}{p}$ and it is well-known that sums of this character over certain intervals, such as $[0,p/2]$ or $[0,p/3]$ are always non-negative. Results of this type are closely related to the behavior of the values of $L$-functions of Dirichlet characters $L(s,\chi)$ at $s=1$. Conversely, positivity of the partial sums
\[
S_{\chi_p}(t)=\sum_{n\leq t}\chi_p(n)
\]
can, in some situations, lead to results on the distribution of zeros of $L$-functions. More precisely, let
\[
f_p(t)=\sum_{n<p}\mlegendre{n}{p}t^n
\]
be the Fekete polynomial. In view of the fact that
\[
L(s,\chi_p)\Gamma(s)=\int_0^{+\infty}\frac{f_p(e^{-t})}{1-e^{-pt}}t^{s-1}dt
\]
for all $\mathrm{Re}\,s>0$, to prove that $L(s,\chi_p)$ has no positive real zeros (which would, for example, imply non-existence of Siegel zeros for $L(s,\chi_p)$) it suffices to show that $f_p(t)$ has no zeros in $(0,1)$. One particularly interesting set of primes, for which $f_p(t)$ indeed has no zeros in $(0,1)$ is primes with $S_{\chi_p}(t)\geq 0$ for all $t$, i.e. primes with non-negative partial sums of $\chi_p(n)$. Following \cite{BCC}, we will denote this set by  $\mathcal L^+$. However, the result by Baker and Montgomery \cite{BM} shows that for most primes $p$ the Fekete polynomial $f_p(t)$ has a lot of zeros between $0$ and $1$. In particular, $\mathcal L^+$ has relative density $0$. Authors do not provide any explicit bound for the rate of decay of relative density, so their results imply the following:
\begin{theorem}
For $x\to +\infty$ we have
\[
|\mathcal L^+\cap [1,x]|=o(\pi(x)).
\]
\end{theorem}

In this paper, we are going to use some more recent theorems on random multiplicative functions to obtain the explicit decay rate and prove the following quantitative estimate.

\begin{theorem}
There is a constant $c>0$ such that
\[
|\mathcal L^+\cap [1,x]|\ll \frac{\pi(x)}{(\ln\ln x)^c}.
\]
One can take $c\approx 0.0368$
\end{theorem}

Borwein, Choi and Coons study the set $\mathcal L^+$ in context of multiplicative functions $\widetilde{\chi_p}$, that are defined as completely multiplicative functions with $\widetilde{\chi_p}(n)=\chi_p(n)$ for $(n,p)=1$ and $\widetilde{\chi_p}(p)=1$. These functions take values in $\{1,-1\}$ and for $p\in \mathcal L^+$ have non-negative partial sums that go to infinity at a logarithmic rate. In view of results of T. Tao on Erd\H{o}s discrepancy problem \cite{Tao}, this makes these functions the most positively biased among multiplicative ``almost counterexamples'' to this conjecture.

\section{Multiplicative Dyck paths}

One of the most natural approaches to our problem of estimating density of $\mathcal L^+$ is to reduce it to some finite combinatorial problem. Our reduction starts with the observation that if $p\in \mathcal L^+$ then the vectors $(1,\mlegendre{n}{p})$ for $n=1,\ldots,p-1$ form a Dyck path of length $p-1$. In other words, these vectors form a lattice path on $\mathbb Z^2$ with each step going either up and right or down and right, which never goes below the line $y=0$ and ends on the aforementioned line. The number of Dyck paths of length $2N$ is counted by Catalan numbers and the probability that a random path of length $2N$ is a Dyck path converges to zero at a rate $\frac{c}{N^{3/2}}$. 

Of course, the length of Dyck paths in question varies. On the other hand, we can consider \emph{incomplete} Dyck paths that do not necessarily end on the line $y=0$, i.e. can have a nonzero total sum. The proportion of incomplete Dyck paths goes to zero slightly slower: at a rate of $O\left(\frac{1}{\sqrt N}\right)$. Incomplete Dyck paths are also called Dyck prefixes and the number of Dyck prefixes of length $N$ is equal to ${N \choose [N/2]}$, see \cite[p. 448, Prop. 9.1.2 and Prop. 9.1.3]{Lw}.

 Moreover, throughout our analysis we missed one important condition which is true for our paths, namely, multiplicativity: for any $a$ and $b$ with $ab\leq p-1$ the $ab$-th step of our path is equal to a pointwise product of $a$-th step and $b$-th step. Let us call such paths \emph{multiplicative paths}. Then the main result that will help us prove Theorem 2 is
\begin{theorem}
Let $m(N)$ be the proportion of incomplete Dyck paths among all multiplicative paths of length $[N]$. Then for $x\geq 8$ we have
\[
|\mathcal L^+\cap [1,x]|\ll \pi(x)m(0.5\ln x).
\]
\end{theorem}

To prove this result we will need the law of quadratic reciprocity and the Brun-Titchmarsh inequality.
\begin{lemma}
For odd primes $p$ and $q$ we have
\[
\mlegendre{p}{q}\mlegendre{q}{p}=(-1)^{(p-1)(q-1)/4}
\]
\end{lemma}
\begin{lemma}
If $q<x$, $(q,a)=1$ and $\pi(x;q,a)$ is the number of primes $p\leq x$ with $p\equiv a \pmod q$ then
\[
\pi(x;q,a)\leq \frac{2x}{\varphi(q)\ln(x/q)}.
\]
\end{lemma}
\begin{proof}
See \cite{MonVau}, p. 90, Theorem 3.9.
\end{proof}

\begin{proof}[of Theorem 3]
Note first that any prime $p\in \mathcal L^+$ with $p>0.5\ln x$ gives us an incomplete multiplicative Dyck path of length $[0.5\ln x]$, namely the path
\[
\ell_p=\left(1,\mlegendre{m}{p}\right)_{m\leq 0.5\ln x}.
\]
Let $M$ be the total number of incomplete multiplicative Dyck paths of this length. Let us show that for any such path $\mu$ there are at most
\[
O\left(\frac{\pi(x)}{2^{\pi(0.5\ln x)}}\right)
\]
primes $p\leq x$ such that $\ell_p=\mu$. Indeed, if $\ell_p=\mu$ then $\mlegendre{q}{p}=\mu_q$ for all primes $q\leq 0.5\ln x$. Let
\[
Q=\prod_{2<q\leq 0.5\ln x}q.
\]
Lemma 1 implies that if $\varepsilon=\pm 1$, $p\equiv \varepsilon \pmod 4$ and $\mlegendre{q}{p}=\mu_q$ then $\mlegendre{p}{q}=\mu_q\varepsilon^{(q-1)/2}$, so for fixed $\varepsilon$ the prime $p$ with $\ell_p=\mu$ lies in one of $\prod_{2<q\leq 0.5\ln x}\frac{q-1}{2}=\varphi(Q)2^{1-\pi(0.5\ln x)}$ residue classes $\mod Q$. Now, by the Prime Number Theorem, we have
\[
Q\ll \exp(0.6\ln x)=x^{3/5}.
\]
Applying the Brun-Titchmarsh inequality to each of the abovementioned residue classes for each choice of $\varepsilon$ and $\mu$, we obtain
\[
|\mathcal L^+\cap [0,x]|\ll \ln x+M\varphi(Q)2^{-\pi(0.5\ln x)}\frac{x}{\varphi(Q)\ln(x/Q)}\ll \frac{M}{2^{\pi(0.5\ln x)}}\pi(x)=\pi(x)m(0.5\ln x),
\]
as required, because there are $2^{\pi(0.5\ln x)}$ multiplicative paths of length $[0.5\ln x]$ and only $M$ of them are incomplete Dyck paths. Also, the $\ln x$ term is absorbed into the second estimate, as $M\geq 1$: there is a multiplicative path that consists of the vectors $(1,1)$.
\end{proof}

One can also notice that the quantity $m(N)$ can be interpreted in terms of random multiplicative functions. More precisely, let
\[
f(2),f(3),\ldots,f(p),\ldots
\]
be the sequence of independent Rademacher random variables indexed by prime numbers. The variable $X$ has a Rademacher distribution if $\mathbb P(X=1)=\mathbb P(X=-1)=\frac12.$ Now we define the random multiplicative function $f(n)$ by the formula
\[
f(n)=\prod_p X_p^{\nu_p(n)},
\]
where $\nu_p(n)$ is the largest $\nu$ with $p^\nu \mid n$. Then it is easy to observe that $m(N)$ is just a probability of the event
\[
X_1\geq 0,
\]
\[
X_1+X_2\geq 0,
\]
\[
\ldots
\]
\[
X_1+\ldots+X_N\geq 0,
\]
because for fixed $X_n$ the probability of the event $f(n)=X_n$ for all $n\leq N$ is equal to $2^{-\pi(N)}$ and to get the total probability we should sum this quantity over all incomplete multiplicative Dyck paths of length $N$.
\section{Proof of Theorem 2}

To obtain our main result, we are going to establish the upper bound
\[
m(N)\ll \frac{1}{(\ln N)^{c+o(1)}}
\]
for all large $N$ in this section. To do so, we will consider the random zeta-function associated to the random multiplicative function $f(n)$ 
\[
\zeta_f(s)=\sum_{n\geq 1}\frac{f(n)}{n^s}
\]
near the line $\mathrm{Re}\,s=\frac12$, which is the abscissa of convergence of this Dirichlet series for almost all $f$. The fact that this series converges for $\mathrm{Re}\,s>\frac12$ for almost all $f$ can be easily deduced from the fact that
\[
M_f(x)=\sum_{n\leq x}f(n)=O(x^{1/2+o(1)})
\]
for almost all $f$ (see \cite{Hal}). Next, we set $\sigma=\sigma_N=1+\frac{3\ln\ln N}{\ln N}$. Using summation by parts we see that for any $t\in \mathbb R$ we have
\[
\zeta_f(\sigma+it)=(\sigma+it)\int_1^{+\infty}\frac{M_f(u)}{u^{\sigma+1+it}}du
\]
If $M_f(u)\geq 0$ for all $u\leq N$ (i.e. inside the event, probability of which we would like to estimate), then
\[
|\zeta_f(\sigma+it)|\leq (1+|t|)\int_1^{N}\frac{M_f(u)}{u^{\sigma+1}}du+(1+|t|)\int_{N}^{+\infty}\frac{|M_f(u)|}{u^{\sigma+1}}du.
\]
Let us denote the last integral by $I=I(N)$ and apply our identity once more to obtain
\[
\zeta_f(\sigma)=\sigma\int_1^{+\infty}\frac{M_f(u)}{u^{\sigma+1}}du\geq \frac12\left(\int_1^N\frac{M_f(u)}{u^{\sigma+1}}-I\right).
\]
Substituting this back into the first inequality, we deduce
\[
|\zeta_f(\sigma+it)|\leq 2(1+|t|)\zeta_f(\sigma)+2(1+|t|)I.
\]
The key to the proof of our result is now the discrepancy between upper bounds for $I$ and $\zeta_f(\sigma)$ for typical $f$ and lower bounds for $\zeta_f(\sigma+it)$. While the former are obtainable by a very elementary probabilistic arguments, we need to invoke results by A.\,Harper \cite{Harp} to prove the latter in the form suitable to our goals.

Taking maximum over $|t|\leq \frac13\ln\ln^3 N-1$, from the inequalities above we get
\begin{lemma}
If $M_f(t)\geq 0$ for all $t\leq N$, then
\[
\sup_{|t|\leq 1/3\ln\ln^3 N-1}|\zeta_f(\sigma+it)|\leq \ln\ln^3 N(\zeta_f(\sigma)+I(N)).
\]
\end{lemma}

To start with elementary arguments, let us compute the expectation of $I(N)$.

\begin{lemma}
We have
\[
\mathbb EI(N)\ll \frac{1}{(\ln N)^{3/2}}
\]
\end{lemma}

To obtain this result, let us first notice that for $u\geq 2$
\[
\mathbb EM_f(u)^2\ll u\ln u.
\]
Indeed,
\[
\mathbb EM_f(u)^2=\mathbb E\left(\sum_{n\leq u}f(u)\right)^2=\sum_{a,b\leq u}\mathbb Ef(a)f(b).
\]
Next, due to complete multiplicativity, $\mathbb Ef(a)f(b)=\mathbb Ef(ab)$, which is equal to $0$ unless $ab$ is a square, in which case it is equal to $1$. The product of two integers is a square if and only if their squarefree parts are equal and the number of $a\leq u$ with a given squarefree part $d$ is $\left[\sqrt{\frac{u}{d}}\right]$, hence
\[
\mathbb EM_f(u)^2=\sum_{d\leq u}\mu(d)^2\left[\sqrt{\frac{u}{d}}\right]^2\leq \sum_{d\leq u}\frac{u}{d}\ll u\ln u,
\]
as needed. Therefore, by Cauchy-Bunyakovsky-Schwarz inequality,
\[
\mathbb E|M_f(u)|\ll \sqrt{u\ln u},
\]
so
\[
\mathbb EI(N)\ll \int_N^{+\infty}\frac{\sqrt{u\ln u}}{u^{\sigma+1}}du.
\]
Changing variables to $u=yN$ and using the inequality $\sqrt{\ln yN}\leq \sqrt{\ln y}+\sqrt{\ln N}$, we get
\[
\mathbb EI(N) \ll N^{-3\ln\ln N/\ln N}\int_1^{+\infty}(\sqrt{\ln y}+\sqrt{\ln N})y^{-1-3\ln\ln N/\ln N}dy\leq
\]
\[
\leq (\ln N)^{-3}\int_0^{+\infty}(\sqrt{\ln y}+\sqrt{\ln N})y^{-1-3\ln\ln N/\ln N}dy.
\]
Substitution $y=\exp(v\ln N/3\ln\ln N)$ gives
\[
\mathbb EI(N)\ll (\ln N)^{-2}\int_0^{+\infty}\sqrt{\ln N}(\sqrt{v}+1)e^{-v}dv\ll (\ln N)^{-3/2},
\]
as needed.

In particular, by Markov's inequality we have
\[
\mathbb P(I(N)\geq 1)\ll \frac{1}{(\ln N)^{3/2}}.
\]
To estimate the probability of the event that $\zeta_f(\sigma)$ is large, we are going to exploit sub-Gaussian nature of its logarithm. More precisely, we are going to use the following result:
\begin{lemma}
Let $s_2,s_3,s_5,\ldots$ be the sequence of real numbers, indexed by prime numbers, and suppose that
\[
\sum_p s_p^2=s
\]
is finite. Then for any $\gamma\geq 0$ we have
\[
\mathbb P(\sum_p s_pf(p)\geq \gamma)\leq \exp\left(-\frac{\gamma^2}{2s}\right).
\]
\end{lemma}
\begin{proof}
Take $T>0$ and compute
\[
\mathbb E\exp(T\sum_p s_pf(p)).
\]
Due to independence of $f(p)$, we have
\[
\mathbb E\exp(T\sum_p s_pf(p))=\prod_p \exp(Ts_pf(p))=\prod_p\cosh(Ts_p)\leq \prod_p\exp(T^2s_p^2/2)=\exp(T^2s/2),
\]
since $\cosh(t)\leq \exp(t^2/2)$ for all $t\in\mathbb R$. Any $f$ with $\sum_p s_pf(p)\geq \gamma$ contributes at least $\exp(T\gamma)$ to the expectation above, so for any $T\geq 0$ we get
\[
\mathbb P(\sum_p s_pf(p)\geq \gamma)\leq \exp(T^2s/2-T\gamma).
\]
Choosing $T=\gamma/s$, we get the desired result.
\end{proof}

Note also that the series converges almost surely in view of Kolmogorov's three series theorem.

Now, let us fix some constant $1/2<\alpha<1$. Using previous lemma and Euler product for $\zeta_f$, we deduce the bound
\begin{lemma}
For $N\to +\infty$ we have
\[
\mathbb P(\zeta_f(\sigma_N)\geq (\ln N)^{\alpha}(\ln\ln N)^{-3}-1)\leq (\ln N)^{-(\alpha-1/2)^2/2+o(1)}.
\]
\end{lemma}
\begin{proof}
Indeed, due to the Euler product for $\zeta_f(\sigma)$, we have
\[
\ln \zeta_f(\sigma)=\sum_p \frac{f(p)}{p^{\sigma}}+\sum_p \frac{f(p)^2}{p^{2\sigma}}+O(1)=\xi+\frac12\ln \zeta(2\sigma)+O(1),
\]
where in the second summand we have the deterministic Riemann zeta-function instead of $\zeta_f$. Next, $\sigma=1/2+\frac{3\ln\ln N}{\ln N}$, so
\[
\ln\zeta(2\sigma)=\ln(2\sigma-1)^{-1}+O(1)=(1+o(1))\ln\ln N.
\]
Therefore, the inequality $\zeta_f(\sigma)\geq  (\ln N)^{\alpha}(\ln\ln N)^{-3}-1$ implies that $\xi\geq (\alpha-1/2-o(1))\ln\ln N$. Notice now that $\xi$ is of the type discussed in Lemma 5 with $s_p=\frac{1}{p^\sigma}$ and
\[
s=\sum_p s_p^2=\ln \zeta(2\sigma)+O(1)=(1+o(1))\ln\ln N.
\]
Thus, from Lemma 5 we obtain
\[
\mathbb P(\xi\geq (\alpha-1/2-o(1))\ln\ln N)\leq \exp\left(-(1+o(1))(\alpha-1/2)^2/2\ln\ln^2 N/\ln\ln N\right)=
\]
\[
=\frac{1}{(\ln N)^{(\alpha-1/2)^2/2+o(1)}},
\]
which concludes the proof.
\end{proof}

Lemmas 3, 4 and 6 together result in the bound for $\zeta(\sigma+it)$ for small $t$.
\begin{corollary}
If $M_f(u)\geq 0$ for all $u\leq N$, then with probability
\[
1-O((\ln N)^{-(\alpha-1/2)^2/2+o(1)})
\]
we have
\[
\sup_{|t|\leq 1/3\ln\ln^3 N-1}|\zeta_f(\sigma+it)|\leq (\ln N)^\alpha.
\]
\end{corollary}

Now, let us choose two parameters $\beta$ and $\gamma$ with $\beta-\gamma>\alpha$. Take $y=\exp((\ln N)^{1-\beta})$. By \cite{Harp}, we have
\begin{lemma}
With probability
\[
1-O\left(\frac{1}{(\ln y)^{1-o(1)}}\right)=1-O\left(\frac{1}{(\ln N)^{1-\beta-o(1)}}\right)
\]
there is a value $1\leq t\leq (\ln\ln N)^3(\ln\ln\ln N)^{-1/3}$ for which we have
\[
\sum_{y<p\leq \exp(\ln N/3\ln\ln N)}\frac{f(p)\cos(t\ln p)}{p^\sigma}\geq (\ln\ln N-\ln\ln y)(1+o(1))=(\beta+o(1))\ln\ln N.
\]
\end{lemma}
To see this, set $B=(\ln\ln N)^3(\ln\ln\ln N)^{-1/3}$ in the first formula on p.\,606 and then apply the last formula of Appendix B. 

Obviously, such a value of $t$ can be chosen to be measurable with respect to $\sigma-$algebra, generated by random variables $f(p)$ for $y<p\leq \exp(\ln N/3\ln\ln N)$. This means that if want to estimate the rest of the series, our Lemma 5 remains applicable due to the independence of $f(p)$. Note also that $|\log\zeta(2\sigma+it)|\ll \ln\ln |t|$ for $|t|\geq 2$, see \cite{MonVau}, Theorem 6.7, hence
\[
\sum_p \frac{\cos^2(t\ln p)}{p^{2\sigma}}=\sum_p \frac{1+\cos(2t\ln p)}{2p^{2\sigma}}=\frac12(\zeta(2\sigma)+\zeta(2\sigma+it))+O(1)=(\frac12+o(1))\ln\ln N.
\]
Therefore, Lemma 5 gives
\[
\mathbb P(-\sum_{p\leq y\text{ or }p>\exp(\ln N/3\ln\ln N)}\frac{f(p)\cos(t\ln p)}{p^\sigma}\geq \gamma \ln\ln N)\leq (\ln N)^{-\gamma^2+o(1)}.
\]
This means that the inequality
\[
\sup_{1\leq |t|\leq 1/3\ln\ln^3 N-1}\sum_p \frac{f(p)\cos(t\ln p)}{p^\sigma}\geq (\beta-\gamma)\ln\ln N
\]
is true with probability
\[
1-O((\ln N)^{-\gamma^2+o(1)})-O((\ln N)^{-\beta+o(1)}).
\]
Due to the fact that
\[
\ln|\zeta_f(\sigma+it)|=\sum_p \frac{f(p)\cos(t\ln p)}{p^\sigma}+\frac12\ln|\zeta(\sigma+it)|+O(1)=\sum_p \frac{f(p)\cos(t\ln p)}{p^\sigma}+o(\ln\ln N)
\]
we see now that for $\delta=\min(\gamma^2,1-\beta)$, the inequality
\[
\sup_{1\leq|t|\leq 1/3\ln\ln^3N-1}|\zeta_f(\sigma+it)|\geq (\ln N)^{\beta-\gamma-o(1)}>(\ln N)^\alpha
\]
holds with probability $1-O((\ln N)^{-\delta})$. On the other hand, Corollary 1 shows that the opposite of this inequality is true with large probability if $M_f(u)\geq 0$ for all $u\leq N$. Applying these estimates together, we conclude that
\[
m(N)\ll (\ln N)^{-\min(1-\beta,\gamma^2,(\alpha-1/2)^2/2)+o(1)}.
\]
Choosing $1-\beta=\gamma^2$ and $\beta=\alpha+\varepsilon+\gamma$ for $\varepsilon\to 0$ we see that for any $\alpha$ between $1/2$ and $1$
\[
m(N)\ll (\ln N)^{-f(\alpha)+o(1)}
\]
for $f(\alpha)=\min((\alpha-1/2)^2/2, 3/2-\alpha-\sqrt{5/4-\alpha})$. This quantity attains its maximal value $c\approx 0.0368$ at $\alpha_{\max}\approx 0.77133$, which concludes the proof of Theorem 1.
\section{Numerical evidence and conclusions}

The estimate of Theorem 2 seems quite weak due to the size of constant $c$. To even recover the obvious bound of $1/2$ for the density of $\mathcal L^+$ one should take $x$ to be around $10^{10^{10^{10}}}$, which is impractical, to say the least. However, numerical experiments suggest that the set $\mathcal L^+$ is, in fact, quite dense. For example, 1000-th term of $\mathcal L^+$ is equal to $55639$ and we have
\[
\ln\ln(55639)\approx 2.39<\frac{\pi(55639)}{1000}=5.646<\ln(55639)\approx 10.92.
\]

One of reasonable conjectures is that for $x\to +\infty$ we have
\[
|\mathcal L^+\cap [1,x]|=\frac{\pi(x)}{(\ln\ln x)^{1-o(1)}}.
\]
However, it is not even clear whether or not the set $\mathcal L^+$ is infinite.

As for the values of $m(N)$, limited data does seem to indicate that multiplicative functions are very biased towards positivity of partial sums in comparison to ``truly random'' functions, see Table 1.

\vspace{0.2cm}
\begin{center}
\begin{table}[h]
\begin{tabular}{|c|c|c|c|c||c|c|c|c|c|}
\hline
$n$ & $p_n$ & $2^n m(p_n)$ & $m(p_n)$ & $m(p_n)\ln p_n$ & $n$ & $p_n$ & $2^n m(p_n)$ & $m(p_n)$ & $m(p_n)\ln p_n$\\
\hline
2 & 3 & 3& 0.75 & 0.824 &11 & 31 & 1020& 0.498& 1.71\\
\hline
3 & 5 & 6& 0.75& 1.207 & 12 & 37 & 1990& 0.486& 1.755\\
\hline
4 & 7 & 10& 0.625& 1.216& 13 & 41 & 3898& 0.476& 1.767\\
\hline
5 & 11  &19 & 0.594& 1.424& 14 & 43 & 7686& 0.469& 1.764\\
\hline
6 & 13  &37 & 0.578& 1.483& 15 & 47 & 14894& 0.456& 1.756\\
\hline
7 & 17  &70 & 0.547& 1.549& 16 & 53 & 29700& 0.453& 1.799\\
\hline
8 & 19  &137 & 0.535& 1.575& 17 & 59 &  57591& 0.439& 1.79\\
\hline
9 & 23&  264 & 0.516& 1.618&18 & 61&  114098& 0.435& 1.788\\
\hline
10 & 29  &521& 0.509& 1.714& 19 & 67 &  225575& 0.430& 1.808\\
\hline
\end{tabular}
\caption{Values of $m(N)$ for the first few prime numbers.}
\end{table}
\end{center}

Randomized experiments for larger values of $n$ seem to indicate that the quantity $m(p_n)\ln p_n$ continues to grow, so it is possible that $m(N)$ goes to zero even slower than $1/\ln N$. These observations show that our results are close to the true order of magnitude of the function in question. Such improvement over previous results was made possible by the treatment of $\zeta_f(s)$ as a function of complex variable, instead of focusing on $\zeta_f(\sigma)$ for real values of $\sigma$ only.


\begin{thebibliography}{99}
\bibitem[1]{BCC} Peter Borwein, Stephen Choi and Michael Coons, ``Completely Multiplicative Functions Taking Values in \{-1, 1\}'', Trans. Amer. Math. Soc., 362 (2010) 6279-6291
\bibitem[2]{BM} Baker R.C., Montgomery H.L. (1990) ``Oscillations of Quadratic L-Functions''. In: Berndt B.C., Diamond H.G., Halberstam H., Hildebrand A. (eds) Analytic Number Theory. Progress in Mathematics, vol 85. Birkh\"auser Boston.
\bibitem[3]{Hal} Hal\'asz, G. (1983). ``On random multiplicative functions''. In Hubert Delange Colloquium
(Orsay, 1982). Publications Math\'{e}matiques d'Orsay 83 74-96. Univ. Paris XI, Orsay.
\bibitem[4]{Harp} Adam J. Harper. ``Bounds on the suprema of Gaussian processes, and omega results for the sum of a random multiplicative function.'' Ann. Appl. Probab. 23 (2) 584 - 616, April 2013.
\bibitem[5]{Lw} Lothaire, M. (2005), ``Applied Combinatorics on Words'' (Encyclopedia of Mathematics and its Applications). Cambridge: Cambridge University Press 
\bibitem[6]{MonVau} Montgomery, H. L. and Vaughan, R. C. (2007). ``Multiplicative Number Theory. I. Classical Theory''. Cambridge Studies in Advanced Mathematics 97. Cambridge Univ. Press, Cambridge.
\bibitem[7]{Tao}  Tao, Terence (2016). ``The Erd\H{o}s discrepancy problem''. Discrete Analysis: 1-29

\end{thebibliography}
\end{document}